\documentclass[11pt]{article}

\usepackage[T1]{fontenc}
\usepackage[a4paper,margin=1in]{geometry}
\usepackage{amsmath,amssymb,amsthm,mathtools}
\usepackage{newtxtext,newtxmath}
\usepackage{needspace}
\usepackage{hyperref}
\hypersetup{
  hidelinks,
  pdftitle={Freeness of Arrangements with Regular Underlying Matroids},
  pdfauthor={Weikang Liang and Suijie Wang},
  pdfsubject={Matroid theory and hyperplane arrangements},
  pdfkeywords={binary matroids, regular matroids, nice partitions,
    supersolvable matroids, free arrangements, chordal graphs}
}
\theoremstyle{plain}
\newtheorem{theorem}{Theorem}[section]

\newtheorem{corollary}[theorem]{Corollary}
\newtheorem{lemma}[theorem]{Lemma}
\theoremstyle{definition}
\newtheorem{definition}[theorem]{Definition}
\theoremstyle{remark}
\newtheorem{remark}[theorem]{Remark}

\newcommand{\FF}{\mathbb{F}}
\newcommand{\rk}{\operatorname{rk}}
\newcommand{\cl}{\operatorname{cl}}
\newcommand{\PG}{\operatorname{PG}}

\renewenvironment{abstract}
  {\begin{list}{}{\setlength{\leftmargin}{\parindent}\setlength{\rightmargin}{\parindent}}%
   \item\relax\small\noindent\textsc{Abstract.}\ }
  {\end{list}}
\renewcommand{\footnoterule}{\kern-3pt\hrule width 1in\kern2.6pt}

\begin{document}
\begin{center}
  {\Large\bfseries
    Freeness of Arrangements with Regular Underlying Matroids\par}

  \vspace{18pt}

  {\normalsize
    Weikang Liang and Suijie Wang\par}
\end{center}

\begingroup
\renewcommand{\thefootnote}{}
\makeatletter
\let\@footnotetext\H@@footnotetext
\long\def\@makefntext#1{\parindent=1.8em\noindent\hspace*{\parindent}#1}
\makeatother
\footnotetext{%
\par
Corresponding author: Suijie Wang.}
\endgroup

\vspace{-6pt}

\begin{abstract}
We classify freeness for finite central arrangements whose underlying
matroids are regular.  Let \(\mathcal A\) be such an arrangement over an
arbitrary field, and put \(M=M(\mathcal A)\).  Then \(\mathcal A\) is free
if and only if \(M\) is supersolvable; equivalently, \(M\) admits a nice
partition; equivalently, \(M\) is the cycle matroid of a chordal simple
graph.  Thus, for arrangements with regular underlying matroids, freeness has
a complete combinatorial classification independent of the base field.

We use Seymour's decomposition theorem for regular matroids to prove that
freeness forces supersolvability.  We also characterize nice partitions of
finite simple binary matroids: a partition is nice if and only if it is
independent and no line is contained in a single block.  Consequently, a
finite loopless binary matroid admits a nice partition if and only if it is
simple and supersolvable.

\end{abstract}

\section{Introduction}

Let \(M\) be a finite simple matroid.  Stanley introduced supersolvable
geometric lattices: \(M\) is supersolvable if its lattice of flats contains a
maximal chain of modular flats.  He proved that supersolvability implies a
factorization of the characteristic polynomial into linear factors over
\(\mathbb Z\)
\cite{stanley-supersolvable}.  Terao later introduced nice partitions:
a partition \(\pi\) of \(E(M)\) is nice if it is independent and, for every
nonempty flat, the induced partition has a singleton block.  Nice partitions
yield factorizations of Orlik--Solomon algebras, from which a linear
factorization of the characteristic polynomial follows
\cite{orlik-terao,terao-factorizations}.  Every
simple supersolvable matroid admits a nice partition, but the converse fails
in general \cite{hoge-roehrle-add-del,orlik-terao}.

Let \(\mathcal A\) be a finite central arrangement over an arbitrary field, with
underlying matroid \(M\).  If \(\mathcal A\) is free, Terao's factorization theorem
\cite{terao-arbitrary-field} shows that its characteristic polynomial factors
into linear factors over \(\mathbb Z\).  In general, freeness of an arrangement and the
existence of a nice partition of its underlying matroid do not imply one
another
\cite{hoge-roehrle-add-del,hoge-roehrle-reflection}.  We prove that, for
arrangements with regular underlying matroids, freeness is equivalent to
supersolvability, to the existence of a nice partition, and to a
chordal-graph description.

For graphic arrangements, the relevant classification is known.  Stanley
proved that the cycle matroid \(M(G)\) of a simple graph \(G\) is
supersolvable if and only if \(G\) is chordal
\cite[Proposition 2.8]{stanley-supersolvable}, while the graphic arrangement
of \(G\) is free if and only if \(G\) is chordal
\cite[Theorem 1.2]{tsujie-modular}.  Ziegler proved that every
supersolvable binary matroid with no \(F_7\)-restriction is graphic
\cite[Theorem 2.7]{ziegler-binary}.  Hence a simple supersolvable regular
matroid is the cycle matroid of a chordal simple graph.

\begin{theorem}
\label{thm:intro-free-classification}
Let \(\mathcal A\) be a finite central arrangement over an arbitrary field,
and suppose that its underlying matroid
\(M=M(\mathcal A)\) is regular.  The following are equivalent.
\begin{enumerate}

  \item \(\mathcal A\) is free.
  \item \(M\) is supersolvable.
  \item \(M\) admits a nice partition.
  \item \(M\cong M(G)\) for some chordal simple graph \(G\).
\end{enumerate}
\end{theorem}

Terao's conjecture asks whether freeness over a fixed field is determined by the
intersection lattice \cite[Conjecture 4.138]{orlik-terao}.  For binary
matroids, Ziegler proved that freeness is invariant among realizations over
fields of a fixed characteristic \cite[Theorem 2.5]{ziegler-free}.  For
regular matroids, Theorem~\ref{thm:intro-free-classification} implies that
this invariance holds across all characteristics.

The two new ingredients are the following.  First, for arrangements with
regular underlying matroids,
\[
 \mathcal A\text{ is free}
 \quad\Longrightarrow\quad
 M(\mathcal A)\text{ is supersolvable}.
\]
Second, for simple binary matroids, every nice partition comes from a
maximal chain of modular flats.  The converse direction is classical: in a
simple supersolvable matroid, a maximal chain of modular flats induces a nice partition
\cite[Proposition 2.67]{orlik-terao}.  The implication from nice partitions
to supersolvability fails for general matroids.  Our second main result
proves the binary statement and characterizes the individual nice partitions.

\begin{theorem}
\label{thm:intro-binary}
Let \(M\) be a finite simple binary matroid of positive rank \(r\), and let
\(\pi=(\pi_1,\ldots,\pi_s)\) be a partition of \(E(M)\) into nonempty
blocks.  The following are equivalent.
\begin{enumerate}
  \item \(\pi\) is nice.
  \item The partition \(\pi\) is independent, and no line of \(M\) is
  contained in a single block.
  \item After reordering, the blocks of \(\pi\) are the successive
  differences of a maximal chain of modular flats.
\end{enumerate}
\end{theorem}

For the freeness-to-supersolvability implication for arrangements with
regular underlying matroids, the
graphic, cographic, and \(R_{10}\) cases are handled separately.  Seymour's
decomposition theory reduces the remaining \(3\)-connected case to exact
\(3\)-separations.  Theorem~\ref{thm:intro-modularization} shows that the
closures \(F_1,F_2\) of the two sides meet in a three-point line that is
modular in both restrictions.  By induction, the two restrictions are
supersolvable, and so is their modular join by
\cite[Theorem 4.7(ii)]{ziegler-binary}.
\begin{theorem}
\label{thm:intro-modularization}
Let \(\mathcal A\) be a finite central free arrangement over an arbitrary
field, with underlying matroid \(M\).
\begin{enumerate}
  \item If \(F_1,F_2\) are flats with
  \(F_1\cup F_2=E(M)\), then they form a modular pair in \(M\).
  \item Suppose that \(M\) is binary and \(3\)-connected, and that
  \((Y_1,Y_2)\) is an exact \(3\)-separation with
  \(|Y_1|,|Y_2|\ge4\).  Put \(F_i=\cl_M(Y_i)\) and
  \(T=F_1\cap F_2\).  Then \(T\) is a three-point line, it is modular in
  both \(M|F_1\) and \(M|F_2\), and
  \[
   M=P_T(M|F_1,M|F_2),
  \]
  where \(P_T\) denotes the modular join, equivalently the generalized
  parallel connection, along \(T\).
  Moreover, \(F_1\) and \(F_2\) are proper, and each \(M|F_i\) is the
  underlying matroid of a free localization of \(\mathcal A\).
\end{enumerate}
\end{theorem}

The regularity hypothesis is essential.  Ziegler's rank-five free
arrangement over \(\FF_2\) has a binary underlying matroid with no
modular hyperplane and hence is not supersolvable
\cite[Example 4.2]{ziegler-free}.  Thus freeness and
supersolvability are not equivalent for arrangements with binary underlying
matroids in general.

Section~\ref{sec:preliminaries} fixes notation.
Section~\ref{sec:nice-binary-regular} proves
Theorem~\ref{thm:intro-binary} and derives the characterization of regular
matroids admitting nice partitions in Corollary~\ref{cor:regular-classification}.
Section~\ref{sec:free-separations} proves
Theorem~\ref{thm:intro-modularization}.
Section~\ref{sec:regular-classification} proves
Theorem~\ref{thm:intro-free-classification}.
Section~\ref{sec:consequences} records consequences and concludes with
Ziegler's example of a free arrangement whose underlying matroid is binary
but not regular.

\section{Preliminaries}
\label{sec:preliminaries}

Our matroid terminology follows \cite{oxley}.  Throughout, all matroids are finite and, unless stated
otherwise, loopless.  For \(X\subseteq E(M)\), the restriction of \(M\) to
\(X\), the deletion of \(X\), and the contraction by \(X\) are denoted by
\(M|X\), \(M\setminus X\), and \(M/X\), respectively; \(M^*\) denotes the
dual of \(M\).  A minor is obtained by a sequence of deletions and
contractions.  The rank, closure, and lattice of flats of \(M\) are denoted
by \(\rk_M\), \(\cl_M\), and \(L(M)\).
Writing \(r=\rk_M(E(M))\) and \(\mu_M\) for the M\"obius function of
\(L(M)\), we use the characteristic polynomial
\[
 \chi(M,t)=
 \sum_{F\in L(M)}\mu_M(\varnothing,F)t^{r-\rk_M(F)}.
\]
We call a rank-two flat a \emph{line}; when \(M\) is simple, its elements
are also called \emph{points}.  A matroid is \emph{Boolean} if its ground
set is independent, and an arrangement is Boolean if its underlying matroid
is Boolean.  A \emph{hyperplane} of a matroid is a maximal proper flat.
For \(Y\subseteq E(M)\), the connectivity function is
\[
\lambda_M(Y)
 =\rk_M(Y)+\rk_M(E(M)\setminus Y)-r.
\]
A partition \((Y_1,Y_2)\) of \(E(M)\) is a
\emph{\(k\)-separation} if \(|Y_1|,|Y_2|\ge k\) and
\(\lambda_M(Y_1)\le k-1\); it is \emph{exact} if
\(\lambda_M(Y_1)=k-1\).  The matroid is \(n\)-connected if it has no
\(k\)-separation for \(k<n\).

\begin{definition}
Let \(\mathbb K\) be a field.  A matroid is
\emph{representable over \(\mathbb K\)} if it is isomorphic to the column
matroid of a matrix over \(\mathbb K\).  It is \emph{binary} if it is
representable over \(\FF_2\), and \emph{regular} if it is representable over
every field.
\end{definition}

Thus every regular matroid is binary.  When \(M\) is simple and binary, we
choose a binary representation and identify its elements with the distinct
nonzero columns over \(\FF_2\).  We write \(F_7=\PG(2,2)\) for the Fano
matroid, \(R_{10}\) for the standard rank-five, ten-element regular matroid
that is neither graphic nor cographic, and \(R_{12}\) for the standard
rank-six, twelve-element regular matroid occurring in Seymour's decomposition
analysis.  For a graph \(G\), \(M(G)\) denotes its cycle
matroid and \(M^*(G)\) its cographic matroid.
A simple graph is \emph{chordal} if it has no induced cycle of length at
least four.

\begin{definition}
Two flats \(X,Y\) of a matroid \(M\) form a \emph{modular pair} if
\[
 \rk_M(X)+\rk_M(Y)
 =
 \rk_M(X\cap Y)+\rk_M(\cl_M(X\cup Y)).
\]
A flat \(X\) is \emph{modular} if it forms a modular pair with every flat
\(Y\) of \(M\).
\end{definition}

\begin{definition}
A rank-\(r\) matroid is
\emph{supersolvable} if it has a chain
\[
 \cl_M(\varnothing)=X_0\subsetneq X_1\subsetneq\cdots
 \subsetneq X_r=E(M)
\]
of modular flats with \(\rk_M(X_i)=i\).
\end{definition}

\begin{definition}
Let \(M\) be a loopless matroid on \(E\).  A partition
\(\pi=(\pi_1,\ldots,\pi_s)\)
of \(E\) is \emph{independent} if every transversal
\(\{e_1,\ldots,e_s\}\), with \(e_i\in\pi_i\), is independent in \(M\).
The partition is \emph{nice} if it is independent and, for every nonempty
flat \(F\), the partition of \(F\) induced by \(\pi\) has a singleton block.
\end{definition}

We now fix the arrangement-theoretic notation used below, following
\cite{orlik-terao}.  Let
\(\mathcal A\) be a finite central arrangement without repeated hyperplanes
in an \(\ell\)-dimensional vector space \(V\) over a field \(\mathbb K\).
For each \(H\in\mathcal A\), choose a linear form
\(\alpha_H\in V^*\) with \(H=\ker(\alpha_H)\), and put
\(S=\operatorname{Sym}(V^*)\cong\mathbb K[x_1,\ldots,x_\ell]\), with its
standard grading.  We write
\(S_+=(x_1,\ldots,x_\ell)=\bigoplus_{d>0}S_d\) for the homogeneous maximal
ideal.
The \emph{underlying matroid} \(M(\mathcal A)\) is the matroid on
\(\mathcal A\) with rank function
\[
 \rk_{M(\mathcal A)}(\mathcal B)
 =
 \dim_{\mathbb K}\operatorname{span}
 \{\alpha_H:H\in\mathcal B\}
 \qquad(\mathcal B\subseteq\mathcal A).
\]
It is simple because the hyperplanes are distinct.  We set
\[
 \rk(\mathcal A)=\rk\!\bigl(M(\mathcal A)\bigr),
 \qquad
 \chi(\mathcal A,t)
 =t^{\ell-\rk(\mathcal A)}
  \chi\!\bigl(M(\mathcal A),t\bigr).
\]
We say that \(\mathcal A\) \emph{realizes} a matroid \(N\) if
\(M(\mathcal A)\cong N\).
The center of \(\mathcal A\) is
\(T(\mathcal A)=\bigcap_{H\in\mathcal A}H\).
The arrangement is \emph{essential} if \(T(\mathcal A)=\{0\}\), equivalently
if \(\rk(\mathcal A)=\ell\).  The arrangement
\[
 \mathcal A^{\mathrm{ess}}
 =
 \{H/T(\mathcal A):H\in\mathcal A\}
\]
in \(V/T(\mathcal A)\) is its \emph{essentialization}.
For a subspace \(X\subseteq V\), the \emph{localization} of \(\mathcal A\) at
\(X\) is \(\mathcal A_X=\{H\in\mathcal A:X\subseteq H\}\).
If \(F\) is a flat of \(M(\mathcal A)\) and
\(X_F=\bigcap_{H\in F}H\), then \(\mathcal A_{X_F}=F\) as a set of
hyperplanes and \(M(\mathcal A_{X_F})=M(\mathcal A)|F\).
We call \(\mathcal A\) \emph{supersolvable} if
\(M(\mathcal A)\) is supersolvable.
For \(H\in\mathcal A\), the \emph{deletion} and \emph{restriction} of
\(\mathcal A\) with respect to \(H\) are
\[
 \mathcal A\setminus\{H\}
 \qquad\text{and}\qquad
 \mathcal A^H
 =\{H\cap K:K\in\mathcal A\setminus\{H\}\},
\]
respectively.  Here \(\mathcal A^H\) is regarded as an arrangement in \(H\),
with repeated intersections identified.

\begin{definition}
The \emph{module of logarithmic derivations} of \(\mathcal A\) is
\[
 D(\mathcal A)
 =\{\theta\in\operatorname{Der}_{\mathbb K}(S):
       \theta(\alpha_H)\in\alpha_H S
       \text{ for every }H\in\mathcal A\}.
\]
Here
\(\operatorname{Der}_{\mathbb K}(S)
=\bigoplus_{i=1}^{\ell}S\,\partial/\partial x_i\).  If
\(\theta=\sum_i f_i\,\partial/\partial x_i\) is homogeneous, its
\emph{polynomial degree} is the common degree of its nonzero coefficients
\(f_i\).
The arrangement \(\mathcal A\) is \emph{free} if \(D(\mathcal A)\) is a
free \(S\)-module.  If \(\mathcal A\) is free, choose a homogeneous
\(S\)-basis \(\theta_1,\ldots,\theta_\ell\) of \(D(\mathcal A)\).
Then \(\exp(\mathcal A)=\{\deg\theta_1,\ldots,\deg\theta_\ell\}\) is the multiset
of \emph{exponents} of \(\mathcal A\); it is independent
of the chosen homogeneous basis.  Its nonzero members are called the
\emph{positive exponents}.
\end{definition}

\section{Nice partitions of binary and regular matroids}
\label{sec:nice-binary-regular}

We prove Theorem~\ref{thm:intro-binary} and then specialize the result to
regular matroids.  We begin with an independent partition of a simple
rank-\(r\) binary matroid into \(r\) blocks.  Using its support digraph,
we construct a maximal chain of flats from the tail unions of the blocks.

\begin{lemma}
\label{lem:standard-coordinate}
Let \(M\) be a simple rank-\(r\) binary matroid and let
\(\pi=(\pi_1,\ldots,\pi_r)\) be an independent partition.  Fix
\(b_i\in\pi_i\).  There is a binary representation of \(M\) in
which \(b_i\) is the \(i\)-th standard basis vector and every
\(x\in\pi_i\) has \(i\)-th coordinate equal to one.
\end{lemma}

\begin{proof}
The set \(B=\{b_1,\ldots,b_r\}\) is an independent \(r\)-set and hence a
basis.  Choose coordinates so that \(B\) is the standard basis of a binary
representation.  For \(x\in\pi_i\), the set
\((B\setminus\{b_i\})\cup\{x\}\) is another independent \(r\)-element
transversal and hence a basis.
Therefore the coefficient of
\(b_i\) in \(x\) is nonzero, and over \(\FF_2\) it equals one.
\end{proof}

Fix the representation in Lemma~\ref{lem:standard-coordinate}.  Define the
\emph{support digraph} \(D_\pi\) on \(\{1,\ldots,r\}\) by putting an arc
\(i\longrightarrow j\)
for \(i\ne j\) if some element of \(\pi_i\) has nonzero \(j\)-th coordinate.

\begin{lemma}
\label{lem:acyclic-support}
The support digraph \(D_\pi\) is acyclic.
\end{lemma}

\begin{proof}
Suppose that \(D_\pi\) has a directed cycle, and choose one of minimum length,
with \(i_1,\ldots,i_m\) distinct:
\[
 i_1\longrightarrow i_2\longrightarrow\cdots
 \longrightarrow i_m\longrightarrow i_1.
\]
The absence of loops in \(D_\pi\) gives \(m\ge2\).  For each \(s\), with
subscripts read modulo \(m\), choose
\(x_s\in\pi_{i_s}\) whose \(i_{s+1}\)-coordinate is one.
By minimality, there is no arc \(i_s\to i_t\) with
\(i_t\notin\{i_s,i_{s+1}\}\).  Indeed, such an arc, together with the directed segment
\[
 i_t\longrightarrow i_{t+1}\longrightarrow\cdots\longrightarrow i_s
\]
of the original cycle, would give a shorter directed cycle.  Hence, among
the coordinates indexed by \(\{i_1,\ldots,i_m\}\), the vector \(x_s\) is
nonzero only in coordinates \(i_s\) and \(i_{s+1}\).

Form a transversal by choosing \(x_s\) from the cycle blocks and \(b_j\) from
all other blocks.  Let \(A\) be the matrix whose columns are the coordinates
of these chosen elements relative to \(B\).  After placing the rows and columns
indexed by \(i_1,\ldots,i_m\) first, \(A\) has block form
\[
 A=
 \begin{bmatrix}
  A_C&0\\
  *&I
 \end{bmatrix}.
\]
Thus \(\det A=\det A_C\).  The principal block \(A_C\) is
\(I_m+P_m\),
where \(P_m\) is the permutation matrix of a directed \(m\)-cycle.  Over
\(\FF_2\),
\[
 (I_m+P_m)\mathbf 1=\mathbf 0,
\]
so \(A_C\) is singular.  Hence the chosen transversal is
dependent, contradicting independence of \(\pi\).
\end{proof}

\begin{corollary}
\label{cor:tail-flats}
After simultaneously reindexing the blocks and the basis vectors according
to a topological ordering of \(D_\pi\), every
\(x\in\pi_i\) has the form
\[
 x=b_i+\sum_{j>i}\varepsilon_j b_j,
 \qquad \varepsilon_j\in\FF_2.
\]
For \(1\le k\le r\), the tail union
\[
 F_k=\pi_k\cup\pi_{k+1}\cup\cdots\cup\pi_r
\]
is a flat of rank \(r-k+1\).
\end{corollary}

\begin{proof}
Take a topological ordering of the acyclic digraph \(D_\pi\).  Then every arc
points from a smaller index to a larger one, so an element of \(\pi_i\) has no
nonzero coordinate in positions \(<i\).
If \(x\in\pi_i\) with \(i\ge k\), then the support of \(x\) is contained in
\(\{i,\ldots,r\}\), and hence
\[
 F_k\subseteq
 E(M)\cap\operatorname{span}_{\FF_2}
 \{b_k,\ldots,b_r\}.
\]
An element in an earlier block \(\pi_i\), \(i<k\), has \(i\)-th coordinate
one and therefore does not lie in this subspace.  Since the blocks partition
\(E(M)\), the reverse inclusion follows.  Consequently,
\[
 F_k=
 E(M)\cap\operatorname{span}_{\FF_2}
 \{b_k,\ldots,b_r\}.
\]
By the chosen representation, the right-hand side is
\(\operatorname{cl}_M(\{b_k,\ldots,b_r\})\).  Thus \(F_k\) is a flat.
Since closure preserves rank and \(\{b_k,\ldots,b_r\}\) is independent,
\(\rk_M(F_k)=r-k+1\).
\end{proof}

The preceding construction assumes that the number of blocks equals the
rank.  The next lemma shows that this equality follows from independence
and the condition that no line is contained in a single block.
We then prove that the resulting chain is modular, establishing
Theorem~\ref{thm:intro-binary}.

\begin{lemma}
\label{lem:block-count}
Let \(M\) be a finite simple binary matroid, and let
\(\pi=(\pi_1,\ldots,\pi_s)\) be an independent partition of \(E(M)\) into
nonempty blocks.  If no line of \(M\) is contained in a single block, then
\(s=\rk(M)\).
\end{lemma}

\begin{proof}
Choose \(b_i\in\pi_i\) for \(1\le i\le s\), fix a binary representation of
\(M\), and set
\[
 B=\{b_1,\ldots,b_s\},
 \qquad
 U=\operatorname{span}_{\FF_2}(B),
 \qquad
 W=\cl_M(B)=E(M)\cap U.
\]
Since \(B\) is an independent transversal, it is a basis of \(U\), and
\(N=M|W\) has rank \(s\).  Since \(b_i\in\pi_i\cap W\), the nonempty sets
\(\pi_i\cap W\) form an independent partition of \(N\).
Apply
Corollary~\ref{cor:tail-flats} to \(N\), using \(B\) as the chosen transversal
and taking coordinates on \(U\) relative to \(B\).  After reindexing the
blocks and the elements of \(B\), every \(w\in\pi_i\cap W\) has the form
\[
 w=b_i+\sum_{j>i}\varepsilon_j b_j,
 \qquad
 \varepsilon_j\in\FF_2.
\]
Thus \(w\in W\) belongs to \(\pi_i\) if and only if its first nonzero
coordinate relative to \(B\) occurs at index \(i\).

Suppose that \(W\subsetneq E(M)\), and choose \(x_0\in E(M)\setminus W\).
We write vector sums in the chosen binary representation.
Given \(x_t\in\pi_{i_t}\setminus W\), the points \(x_t\) and \(b_{i_t}\) lie in
the same block and span the line \(L_t=\cl_M\{x_t,b_{i_t}\}\).
By hypothesis, this line is not contained in \(\pi_{i_t}\).  A line of a simple binary
matroid has at most three points, so it follows that \(L_t\) has the third
point \(x_{t+1}=x_t+b_{i_t}\), which belongs to a block different from
\(\pi_{i_t}\).  Moreover, \(x_{t+1}\notin W\), since otherwise
\(x_t=x_{t+1}+b_{i_t}\) would belong to \(W\).
Iterating this construction gives a sequence in the finite set
\(E(M)\setminus W\).
After discarding an initial segment and cyclically reindexing, it contains a
cycle
\[
 x_0,x_1,\ldots,x_{m-1},x_m=x_0
\]
with \(x_0,\ldots,x_{m-1}\) distinct.  Summing the transition equations
over \(\FF_2\) gives
\[
 \sum_{t=0}^{m-1}b_{i_t}=0.
\]
Since \(B\) is independent, every index occurring in the cycle occurs an
even number of times.

Let \(a\) be the smallest index occurring in the cycle.  After another
cyclic reindexing, choose consecutive occurrences
\(i_u=i_v=a\), with \(u<v\), such that \(a\) does not occur among
\(i_{u+1},\ldots,i_{v-1}\).  Then
\[
 w=x_u+x_v=\sum_{t=u}^{v-1}b_{i_t}
\]
has its first nonzero coordinate relative to \(B\) at index \(a\).
The points \(x_u\) and \(x_v\)
are distinct and both lie in \(\pi_a\).  Their line is not contained in
\(\pi_a\), so its third point \(w\) belongs to \(E(M)\setminus\pi_a\).  On
the other hand, \(w\in W\), so the preceding description of the blocks gives
\(w\in\pi_a\), a contradiction.  Hence \(W=E(M)\), and therefore
\(\rk(M)=\rk(W)=|B|=s\).
\end{proof}

\begin{proof}[Proof of Theorem~\ref{thm:intro-binary}]
Suppose first that \(\pi\) is nice.  Independence is part of the definition.
If a line \(L\) were contained in one block, then the partition induced on
\(L\) would consist of the single nonsingleton block \(L\), contradicting the
singleton axiom.  Thus (1) implies (2).

Assume (2).  Lemma~\ref{lem:block-count} gives \(s=r\).  By
Corollary~\ref{cor:tail-flats}, reorder the blocks so that
\[
 \varnothing=F_{r+1}\subsetneq F_r\subsetneq\cdots
 \subsetneq F_2\subsetneq F_1=E(M),
 \qquad
 F_k=\bigcup_{i=k}^{r}\pi_i,
\]
is a maximal chain of flats.  The rank formula in
Corollary~\ref{cor:tail-flats} gives
\(\rk(F_k)=r-k+1\), and disjointness of the blocks gives
\(F_k\setminus F_{k+1}=\pi_k\).
Fix \(1\le k<r\) and set \(N_k=M|F_k\).  Since \(F_{k+1}\) is a flat of
\(N_k\) of rank \(\rk(N_k)-1\), it is a hyperplane of \(N_k\).
Every line \(L\) of \(N_k\) is
also a line of \(M\).  Indeed, restriction preserves the rank of subsets of
\(F_k\), while
\[
   L=\cl_{N_k}(L)=\cl_M(L)\cap F_k=\cl_M(L),
\]
 where the last equality holds because \(F_k\) is a flat of \(M\) containing
 \(L\).  If this line were
disjoint from \(F_{k+1}\), it would be contained in
\(F_k\setminus F_{k+1}=\pi_k\),
contrary to (2).  Thus \(F_{k+1}\) meets every line of \(N_k\).
In a loopless matroid, a hyperplane is modular if and only if it meets every
line \cite[Corollary 6.9.3]{oxley}, so \(F_{k+1}\) is modular in \(N_k\).
If \(X\) is modular in \(M\) and \(Y\) is modular in \(M|X\), then \(Y\)
is modular in \(M\) \cite[Proposition 6.9.7]{oxley}.
Since \(F_1=E(M)\) is modular and \(F_{k+1}\) is modular in \(M|F_k\)
for \(1\le k<r\), applying this implication successively shows that
\(F_2,\ldots,F_r\) are modular in \(M\).  The
remaining endpoint \(F_{r+1}=\varnothing\) is modular as well.
Thus the displayed chain is modular, and its successive differences are the
blocks of \(\pi\).  This proves (3).

Finally, the successive differences of a maximal chain of modular flats in a
simple supersolvable matroid form a nice partition
\cite[Proposition 2.67]{orlik-terao}, so (3) implies (1).
\end{proof}

\Needspace{8\baselineskip}
\begin{corollary}
\label{cor:binary-equivalence}
For a finite loopless binary matroid \(M\), the following are equivalent.
\begin{enumerate}
  \item \(M\) admits a nice partition.
  \item \(M\) is simple and supersolvable.
\end{enumerate}
If these conditions hold, every nice partition of \(M\) is induced by a
maximal chain of modular flats.
\end{corollary}

\begin{proof}
If \(M\) has rank zero, then looplessness forces \(E(M)=\varnothing\); the
empty partition is nice and the one-term chain \((\varnothing)\) is a maximal
chain of modular flats.
Assume that \(M\) has positive rank, and let
\(\pi=(\pi_1,\ldots,\pi_s)\) be a nice partition.  If two parallel elements
belonged to different blocks, they could be extended to a dependent
transversal, contradicting the independence of \(\pi\).  Hence every parallel
class is contained in a single block.  A nontrivial parallel class is a
rank-one flat, and the partition induced on this flat would consist of a
single nonsingleton block, contradicting the singleton condition.  Thus \(M\)
has no parallel elements and, since it is loopless, \(M\) is simple.
Theorem~\ref{thm:intro-binary} now shows that \(\pi\) is induced by a maximal
chain of modular flats.  In particular, \(M\) is supersolvable.  Conversely,
every simple supersolvable matroid admits a nice partition
\cite[Proposition 2.67]{orlik-terao}.
\end{proof}

\begin{remark}
There are arrangements that admit nice partitions but whose underlying
matroids are not supersolvable
\cite{hoge-roehrle-add-del,hoge-roehrle-reflection}.  Thus the binary
hypothesis in Corollary~\ref{cor:binary-equivalence} is essential.
\end{remark}

We now specialize the preceding characterization to regular matroids.
Regular matroids are binary and exclude both \(F_7\) and \(F_7^*\) as minors
\cite[Theorem 6.6.6]{oxley}.  Combining
Corollary~\ref{cor:binary-equivalence} with
\cite[Theorem 2.7]{ziegler-binary} and
\cite[Proposition 2.8]{stanley-supersolvable} yields the following
chordal-graph classification.

\begin{corollary}
\label{cor:regular-classification}
Let \(M\) be a finite loopless regular matroid.  Then the following conditions
are equivalent.
\begin{enumerate}
  \item \(M\) admits a nice partition;
  \item \(M\) is simple and supersolvable;
  \item \(M\cong M(G)\) for some chordal simple graph \(G\).
\end{enumerate}
\end{corollary}

\begin{proof}
Corollary~\ref{cor:binary-equivalence} gives the equivalence of (1) and (2).
Assume (2).  Since \(M\) is regular, it has no \(F_7\)-minor.  Hence \(M\)
is a simple supersolvable binary matroid with no \(F_7\)-restriction.  By
\cite[Theorem 2.7]{ziegler-binary}, \(M\cong M(G)\) for a simple graph
\(G\), and \cite[Proposition 2.8]{stanley-supersolvable} shows that \(G\) is
chordal.
Thus (2) implies (3).
Conversely, if \(M\cong M(G)\) for a chordal simple graph \(G\), then \(M\)
is simple and supersolvable by
\cite[Proposition 2.8]{stanley-supersolvable}.  Thus (3) implies (2).
\end{proof}

We now turn to free arrangements, beginning with the exact-separation
results needed to prove Theorem~\ref{thm:intro-free-classification}.

\section{Exact separations of free arrangements}
\label{sec:free-separations}

Let \(\mathcal A\) be a finite central free arrangement over a field
\(\mathbb K\), and put \(M=M(\mathcal A)\).  For \(e\in E(M)\), write
\(H_e=\ker(\alpha_e)\) for the corresponding hyperplane of \(\mathcal A\).
For a flat \(F\) of \(M\), put \(X_F=\bigcap_{e\in F}H_e\) and define,
over the fixed ambient polynomial ring \(S\),
\[
 D_F:=D(\mathcal A_{X_F})
 =\{\theta\in\operatorname{Der}_{\mathbb K}(S):
       \theta(\alpha_e)\in\alpha_e S\text{ for every }e\in F\}.
\]
Since \(F\) is a flat, \(\mathcal A_{X_F}=\{H_e:e\in F\}\).
Each \(D_F\) is free over \(S\), since localizations of free arrangements are free
\cite[Theorem 1.7(i)]{ziegler-free}.

\begin{lemma}
\label{lem:elementary-finite-length}
Let \(R=k[x_1,\ldots,x_\ell]\), where \(k\) is a field, and let
\(3\le r\le\ell\).
Suppose that
\[
 0\longrightarrow P\xrightarrow{\phi}Q
 \xrightarrow{\psi}L\longrightarrow N\longrightarrow0
\]
is exact, with \(P,Q,L\) finite free \(R\)-modules.  If
\(N_{\mathfrak p}=0\) for every prime ideal \(\mathfrak p\) of \(R\)
not containing \((x_1,\ldots,x_r)\), then \(N=0\).
\end{lemma}

\begin{proof}
Suppose \(N\ne0\).  Since \(R\) is Noetherian, choose \(0\ne c\in N\)
such that \(\mathfrak p=\operatorname{Ann}_R(c)\) is maximal among
the annihilators of nonzero elements of \(N\).
If \(ab\in\mathfrak p\) and \(b\notin\mathfrak p\), then \(bc\ne0\),
so maximality gives \(\operatorname{Ann}_R(bc)=\mathfrak p\) and
\(a\in\mathfrak p\).  Thus \(\mathfrak p\) is prime.
Since \(c/1\ne0\) in \(N_{\mathfrak p}\), the hypothesis implies
\((x_1,\ldots,x_r)\subseteq\mathfrak p\), hence
\(x_1c=x_2c=x_3c=0\).  Choose \(v\in L\) mapping to \(c\).
Exactness gives
\(u_i\in Q\) with \(\psi(u_i)=x_i v\) for \(i=1,2,3\).
Since \(x_j u_i-x_i u_j\in\ker\psi\), choose \(w_{ij}\in P\) with
\(\phi(w_{ij})=x_j u_i-x_i u_j\) for \(i<j\), and injectivity of
\(\phi\) gives
\[
 x_3w_{12}-x_2w_{13}+x_1w_{23}=0.
\]
In a basis of \(P\), setting \(x_1=x_2=0\) in this relation shows
that every coordinate of \(w_{12}\) belongs to \((x_1,x_2)\).
Thus \(w_{12}=x_2t+x_1s\) for some \(t,s\in P\).
Applying \(\phi\) gives
\(x_2(u_1-\phi(t))=x_1(u_2+\phi(s))\).
In a basis of \(Q\), setting \(x_1=0\) shows that every coordinate
of \(u_1-\phi(t)\) is divisible by \(x_1\), so
\(u_1=\phi(t)+x_1u\) for some \(u\in Q\).
Applying \(\psi\) yields \(x_1v=x_1\psi(u)\).
Since \(L\) is free, cancellation of \(x_1\) gives \(v=\psi(u)\),
contradicting \(c\ne0\).
\end{proof}

\begin{proof}[Proof of Theorem~\ref{thm:intro-modularization}(1)]
We prove \(D_{F_1}+D_{F_2}=D_T\), where \(T=F_1\cap F_2\),
by induction on \(r=\rk(\mathcal A)\).  Suppose \(r\le2\).
If one \(F_i=E(M)\), the other equals \(T\), and the equality is
immediate.  Otherwise, since \(M\) is simple, each proper flat contains
at most one element.  Thus \(F_1=\{e_1\}\) and \(F_2=\{e_2\}\)
with \(e_1\ne e_2\), so \(T=\varnothing\).  Choose coordinates with
\(\alpha_{e_1}=x_1\) and \(\alpha_{e_2}=x_2\).  Then
\(\partial/\partial x_1\in D_{F_2}\), while
\(\partial/\partial x_j\in D_{F_1}\) for \(j\ge2\).
These derivations generate \(\operatorname{Der}_{\mathbb K}(S)=D_T\),
proving the equality.
Assume \(r\ge3\) and that the assertion holds for all free arrangements
of rank less than \(r\).  Put \(C=D_T/(D_{F_1}+D_{F_2})\).
Since \(F_1\cup F_2=E(M)\), we have
\(D_{F_1}\cap D_{F_2}=D(\mathcal A)\).
The maps \(\theta\mapsto(\theta,\theta)\) and
\((\theta_1,\theta_2)\mapsto\theta_1-\theta_2\) give the exact sequence
\begin{equation}
 0\longrightarrow D(\mathcal A)\longrightarrow D_{F_1}\oplus D_{F_2}
 \longrightarrow D_T\longrightarrow C\longrightarrow0.
\label{eq:localized-derivation-cover}
\end{equation}
The first three terms are free.
Choose coordinates so that the defining forms span
\(\langle x_1,\ldots,x_r\rangle\), and put
\(I=(x_1,\ldots,x_r)\subseteq S\).
Let \(\mathfrak p\) be a prime ideal of \(S\) not containing \(I\).
Then \(H=\{e\in E(M):\alpha_e\in\mathfrak p\}\) is a flat of \(M\)
and is proper because \(I\not\subseteq\mathfrak p\).
The arrangement \(\mathcal A_{X_H}\) is free of rank less than \(r\),
and \(H\cap F_1\) and \(H\cap F_2\) are flats of \(M|H\) with union \(H\).
The induction hypothesis therefore gives
\(D_{H\cap F_1}+D_{H\cap F_2}=D_{H\cap T}\).
Since \(\alpha_e\) is a unit in \(S_{\mathfrak p}\) for \(e\notin H\),
\((D_F)_{\mathfrak p}=(D_{F\cap H})_{\mathfrak p}\) for
\(F=F_1,F_2,T\).  Hence
\[
 C_{\mathfrak p}
 =\frac{(D_{H\cap T})_{\mathfrak p}}
 {(D_{H\cap F_1})_{\mathfrak p}+(D_{H\cap F_2})_{\mathfrak p}}=0.
\]
Lemma~\ref{lem:elementary-finite-length}, applied over \(S\) to
\eqref{eq:localized-derivation-cover}, now gives \(C=0\),
completing the induction.

Put \(\mathcal A_i=\mathcal A_{X_{F_i}}\)
and \(\mathcal A_T=\mathcal A_{X_T}\).  We obtain
\begin{equation}
 0\longrightarrow D(\mathcal A)
 \longrightarrow D(\mathcal A_1)\oplus D(\mathcal A_2)
 \longrightarrow D(\mathcal A_T)\longrightarrow0.
\label{eq:derivation-cover-exact}
\end{equation}
This sequence splits as graded modules because \(D(\mathcal A_T)\) is
graded free.  For any subarrangement \(\mathcal B\), its degree-zero
derivations are constant vector fields in \(T(\mathcal B)\), so
\(\dim_{\mathbb K}D(\mathcal B)_0=\ell-\rk(M(\mathcal B))\).
Taking degree zero in \eqref{eq:derivation-cover-exact} therefore gives
\(\rk_M(F_1)+\rk_M(F_2)=\rk(M)+\rk_M(T)\).
\end{proof}

Let \(N_i\) be matroids on \(E_i\) with a common restriction to
\(X=E_1\cap E_2\), a modular flat of each.  The modular join
\(P_X(N_1,N_2)\) is the matroid on \(E_1\cup E_2\) whose flats are exactly
the sets \(Z\) for which \(Z\cap E_i\) is a flat of \(N_i\) for \(i=1,2\).
Both \(E_i\) are modular flats of this matroid
\cite[Definition 3.1]{ziegler-binary}.

\begin{lemma}
\label{lem:recognize-modular-join}
Let \(F_1,F_2\) be a modular pair of flats of a representable matroid
\(M\), with \(F_1\cup F_2=E(M)\).
If \(T=F_1\cap F_2\) is modular in each \(M|F_i\), then
\(M=P_T(M|F_1,M|F_2)\).
\end{lemma}

\begin{proof}
Fix a representation \((v_e)_{e\in E(M)}\), and write \(\langle X\rangle\)
for the span of the vectors indexed by \(X\).
Since \(F_1\cup F_2=E(M)\), the dimension formula and the modular-pair
identity give \(\dim(\langle F_1\rangle\cap\langle F_2\rangle)=\rk_M(T)\).
As \(\langle T\rangle\) is contained in this intersection, we obtain
\(\langle F_1\rangle\cap\langle F_2\rangle=\langle T\rangle\).
For \(Z\subseteq E(M)\), put \(Z_i=Z\cap F_i\).  If \(Z\) is a flat,
each \(Z_i\) is a flat of \(M|F_i\).  Conversely, suppose both \(Z_i\)
are flats and take \(e\in F_1\cap\cl_M(Z)\).  Write \(v_e=z_1+z_2\)
with \(z_i\in\langle Z_i\rangle\).  We have
\(z_2=v_e-z_1\in\langle F_1\rangle\cap\langle F_2\rangle=\langle T\rangle\).
Since \(T\) and \(Z_2\) form a modular pair in \(M|F_2\), the dimension
formula gives
\[
 z_2\in\langle T\rangle\cap\langle Z_2\rangle
 =\langle T\cap Z_2\rangle\subseteq\langle Z_1\rangle,
\]
since \(T\cap Z_2\subseteq Z_1\).
Hence \(v_e\in\langle Z_1\rangle\), so \(e\in Z_1\), since \(Z_1\)
is a flat of \(M|F_1\).  Interchanging the indices gives
\(F_2\cap\cl_M(Z)\subseteq Z_2\).
Since \(F_1\cup F_2=E(M)\) and \(Z_1\cup Z_2=Z\), we have
\(\cl_M(Z)\subseteq Z\).  Together with \(Z\subseteq\cl_M(Z)\), this gives
\(\cl_M(Z)=Z\).  Thus \(M\) and \(P_T(M|F_1,M|F_2)\) have the same flats.
\end{proof}

\begin{proof}[Proof of Theorem~\ref{thm:intro-modularization}(2)]
Since \(Y_1\cup Y_2=E(M)\) and \(Y_i\subseteq F_i\), we have
\(F_1\cup F_2=E(M)\).  Part~(1) therefore shows that \(F_1,F_2\)
form a modular pair.  The hypothesis that \((Y_1,Y_2)\) is an exact
\(3\)-separation means that
\(\rk_M(Y_1)+\rk_M(Y_2)-\rk(M)=2\).
Since taking closures preserves rank, the modular-pair equality gives
\[
 \rk_M(T)=\rk_M(F_1)+\rk_M(F_2)-\rk(M)
 =\rk_M(Y_1)+\rk_M(Y_2)-\rk(M)=2.
\]
In a simple binary matroid, a rank-two set has at most three elements,
since a two-dimensional vector space over \(\FF_2\) has only three
nonzero vectors.  Thus \(T\) has two or three elements.
Both \(F_i\) are proper.  Indeed, if \(F_i=E(M)\), then
\(\rk_M(Y_i)=\rk(M)\), so the exact \(3\)-separation equality gives
\(\rk_M(Y_j)=2\) for \(j\ne i\).  This implies \(|Y_j|\le3\),
contrary to \(|Y_j|\ge4\).

Suppose \(T=\{p,q\}\).  Put \(\mathcal A_i=\mathcal A_{X_{F_i}}\)
and \(\mathcal A_T=\mathcal A_{X_T}\).
These arrangements are free, since they are localizations of \(\mathcal A\).
Let \(m_1(\mathcal B)\) denote the multiplicity of the exponent \(1\)
of a free arrangement \(\mathcal B\).
For a connected matroid \(N\) with at least two elements, \(1\) is a
simple root of \(\chi(N,t)\) \cite[Theorem~II]{crapo-beta}.
Terao's factorization theorem states that
\(\chi(\mathcal B,t)=\prod_{j=1}^{\ell}(t-e_j)\) for a free arrangement
\(\mathcal B\) with exponents \(e_1,\ldots,e_\ell\)
\cite{terao-arbitrary-field}.  Consequently,
\(m_1(\mathcal B)=1\) whenever \(M(\mathcal B)\) is connected and has
at least two elements.  The \(3\)-connectivity of \(M\) implies that
\(M\) is connected, so \(m_1(\mathcal A)=1\).
Since \(T\) consists of two independent elements, \(\mathcal A_T\)
is Boolean of rank two and \(m_1(\mathcal A_T)=2\).
The graded splitting of \eqref{eq:derivation-cover-exact} identifies
\(D(\mathcal A_1)\oplus D(\mathcal A_2)\) with
\(D(\mathcal A)\oplus D(\mathcal A_T)\).  Comparing the multiplicities
of the exponent \(1\) gives
\[
 m_1(\mathcal A_1)+m_1(\mathcal A_2)
 =m_1(\mathcal A)+m_1(\mathcal A_T)=3.
\]
If both \(M|F_i\) were connected, each term on the left would be \(1\),
since \(|F_i|\ge|Y_i|\ge4\).  Thus at least one restriction is
disconnected; after interchanging the indices if necessary, assume
that it is \(M|F_1\).

In the representation by the defining forms, put
\(\langle X\rangle=\operatorname{span}_{\mathbb K}\{\alpha_e:e\in X\}\).
The modular-pair equality and the dimension formula give
\(\langle F_1\rangle\cap\langle F_2\rangle=\langle T\rangle\).
Let \(C\) be a connected component of \(M|F_1\) and take
\(z\in\langle C\rangle\cap\langle E(M)\setminus C\rangle\).
Since \(E(M)\setminus C\subseteq(F_1\setminus C)\cup F_2\), write
\(z=z_1+z_2\), with \(z_1\in\langle F_1\setminus C\rangle\) and
\(z_2\in\langle F_2\rangle\).  Since \(z,z_1\in\langle F_1\rangle\),
we have \(z_2=z-z_1\in\langle F_1\rangle\cap\langle F_2\rangle
=\langle T\rangle\).
The connected components of \(M|F_1\) give the direct decomposition
\(\langle F_1\rangle=\langle C\rangle\oplus\langle F_1\setminus C\rangle\).
Let \(\pi\) be the projection onto \(\langle C\rangle\) along
\(\langle F_1\setminus C\rangle\).  It fixes \(\alpha_e\) for
\(e\in C\) and maps \(\alpha_e\) to zero for \(e\in F_1\setminus C\),
so \(\pi(\langle T\rangle)=\langle C\cap T\rangle\).
Since \(z\in\langle C\rangle\) and
\(z_1\in\langle F_1\setminus C\rangle\), applying \(\pi\) to
\(z=z_1+z_2\) gives
\(z=\pi(z_2)\in\langle C\cap T\rangle\).
Thus
\(\lambda_M(C)=\dim_{\mathbb K}(\langle C\rangle\cap
\langle E(M)\setminus C\rangle)\le\rk_M(C\cap T)\).
If \(C\cap T=\varnothing\), this gives \(\lambda_M(C)=0\).
Since \(C\) is nonempty and is contained in the proper flat \(F_1\),
\((C,E(M)\setminus C)\) would be a \(1\)-separation, contrary to the
connectedness of \(M\).  Therefore every component meets \(T\).
As \(M|F_1\) is disconnected and \(|T|=2\), it has exactly two
components, each containing one element of \(T\).
Since \(|F_1|\ge|Y_1|\ge4\), one
component \(C\) has at least two elements.  Its complement contains the
other element of \(T\) and, since \(F_1\) is proper, an element outside
\(F_1\), so it also has at
least two elements.  The inequality \(\lambda_M(C)\le1\) then makes
\((C,E(M)\setminus C)\) a \(2\)-separation, contradicting
\(3\)-connectivity.  Hence \(|T|=3\).

Now choose a representation \((v_e)_{e\in E(M)}\) of \(M\) over
\(\FF_2\), and put \(\langle X\rangle=
\operatorname{span}_{\FF_2}\{v_e:e\in X\}\).
Since \(T\) has rank two and three elements, its representing vectors
are all the nonzero vectors of \(\langle T\rangle\).
For any flat \(Z\) of \(M\), every nonzero vector in
\(\langle T\rangle\cap\langle Z\rangle\) is therefore \(v_e\) for some
\(e\in T\).  As \(v_e\in\langle Z\rangle\) and \(Z\) is a flat,
\(e\in Z\).  Consequently
\(\langle T\rangle\cap\langle Z\rangle=\langle T\cap Z\rangle\).
The dimension formula gives
\(\rk_M(T)+\rk_M(Z)=\rk_M(T\cup Z)+\rk_M(T\cap Z)\), so \(T\)
is modular in \(M\).  Since \(F_i\) is a flat of \(M\), every flat of
\(M|F_i\) is also a flat of \(M\), and ranks are unchanged by restriction.
Thus the same equality shows that \(T\) is modular in each \(M|F_i\).
Lemma~\ref{lem:recognize-modular-join} now gives
\(M=P_T(M|F_1,M|F_2)\).  Each \(M|F_i\) is realized by the free
localization \(\mathcal A_{X_{F_i}}\).
\end{proof}

\begin{remark}
\label{rem:nondegenerate-three-separation}
The size assumptions \(|Y_1|,|Y_2|\ge4\) cannot be omitted, even when
\(M\) is \(3\)-connected.  For example, the graphic arrangement of \(K_4\) is free,
while its underlying matroid \(M(K_4)\) has an exact \(3\)-separation whose
two sides are a triangle and its three-edge complement.  The three-edge
complement spans \(M(K_4)\), so its closure is the whole ground set.
Thus this separation does not yield a decomposition into restrictions to
two proper flats.  By contrast, the separation used below is
induced by the standard \(3\)-separation of an \(R_{12}\)-minor.  It has at
least six elements on each side and satisfies the hypotheses of
Theorem~\ref{thm:intro-modularization}(2).
\end{remark}

\section{Freeness of arrangements with regular underlying matroids}
\label{sec:regular-classification}

We now prove the main classification by combining the modular-join result of
Section~\ref{sec:free-separations} with the graphic, cographic, and
\(R_{10}\) cases.  We first record two arrangement-theoretic facts used in
these cases.  After essentialization, a free arrangement of rank \(r\) has
\(\chi(\mathcal A,t)=\chi(M(\mathcal A),t)\).  Hence Terao's factorization
theorem over an arbitrary field gives
\begin{equation}
 \chi(M(\mathcal A),t)=\prod_{i=1}^{r}(t-e_i),
 \qquad e_i\in\mathbb Z_{>0},
\label{eq:free-integral-roots}
\end{equation}
where \(e_1,\ldots,e_r\) are the positive exponents of \(\mathcal A\)
\cite{terao-arbitrary-field}; in particular, all roots of
\(\chi(M(\mathcal A),t)\) are positive integers.  We also use the fact that, for a binary
matroid, arrangements
realizing it over fields of the same characteristic are simultaneously free
or nonfree \cite[Theorem 2.5]{ziegler-free}.

\begin{lemma}
\label{lem:free-basic-blocks}
Every simple graphic or cographic matroid realized by a free arrangement is
supersolvable.  The matroid \(R_{10}\) is not realized by any free
arrangement over any field.
\end{lemma}

\begin{proof}
Let \(N\) be a simple graphic or cographic matroid, or let \(N=R_{10}\), and
suppose that an arrangement realizing \(N\) is free.
If \(N=M(G)\) is graphic, then
\cite[Theorem 2.5]{ziegler-free} implies that the standard graphic
arrangement over the field of the given realization,
\[
 \mathcal A_G=\{x_i-x_j=0:\{i,j\}\in E(G)\},
\]
is free.  Moreover,
\cite[Theorem 1.2]{tsujie-modular} gives
\[
 \mathcal A_G\text{ is free}
 \quad\Longleftrightarrow\quad
 G\text{ is chordal}
\]
over an arbitrary field.  The cycle matroid of a chordal graph is
supersolvable \cite[Proposition 2.8]{stanley-supersolvable}.
Hence \(N\) is supersolvable.
Suppose next that \(N\) is cographic.  Equation
\eqref{eq:free-integral-roots} shows that every root of \(\chi(N,t)\) is an
integer.  Kung and Royle proved for simple cographic matroids that this is
equivalent to supersolvability; equivalently, \(N\) is the cycle matroid of a
planar chordal graph \cite[Theorem 6.2]{kung-royle}.
Finally, the characteristic polynomial of \(R_{10}\) is
\[
 \chi(R_{10},t)
 =t^5-10t^4+45t^3-105t^2+120t-51
 =(t-1)q(t),
\]
where
\[
 q(t)=t^4-9t^3+36t^2-69t+51.
\]
This follows from the Tutte polynomial of \(R_{10}\) in
\cite[Equation (4.35)]{merino-tutte}, using
\(\chi(N,t)=(-1)^{\rk(N)}T_N(1-t,0)\).  Since \(q\) is monic, an integral
root would divide \(51\), and would therefore be odd.  For every odd integer
\(a\), however,
\[
 q(a)\equiv 2a\not\equiv0\pmod 8.
\]
Thus \(q\) has no integral root, and \(\chi(R_{10},t)\) does not split into
linear factors over \(\mathbb Z\), contradicting
\eqref{eq:free-integral-roots}.  Hence \(R_{10}\) cannot be realized by a
free arrangement.
\end{proof}

\begin{proof}[Proof of Theorem~\ref{thm:intro-free-classification}]
We first prove (1) \(\Rightarrow\) (2) by induction on \(|E(M)|\).
Essentialize \(\mathcal A\).  If \(\rk(M)=0\), then simplicity gives
\(E(M)=\varnothing\), and the one-element lattice of flats is
supersolvable.  If \(|E(M)|\le3\), then either \(\rk(M)\le2\), when every maximal chain is
modular, or \(\rk(M)=3\), when simplicity forces \(M\) to be Boolean.
We may therefore assume \(|E(M)|\ge4\).

A modular join of two supersolvable binary matroids with no
\(F_7\)-restriction is supersolvable
\cite[Theorem 4.7(ii)]{ziegler-binary}.  Every restriction of \(M\) is
regular, hence binary with no \(F_7\)-restriction.  The induction hypothesis
applies whenever a restriction \(M|F\) to a proper flat \(F\) is the
underlying matroid of a free localization.

If \(M\) is disconnected, choose nonempty unions \(E_1,E_2\) of components
that partition \(E(M)\).  They are proper flats and
\[
 M=M|E_1\oplus M|E_2=P_{\varnothing}(M|E_1,M|E_2).
\]
The corresponding localizations are free and have smaller ground sets.
Induction makes both restrictions supersolvable, and
\cite[Theorem 4.7(ii)]{ziegler-binary}, applied over the empty modular flat,
shows that \(M\) is supersolvable.

Suppose that \(M\) is connected but not \(3\)-connected.  Since \(M\) is
simple and \(|E(M)|\ge4\), it has an exact \(2\)-separation
\((Y_1,Y_2)\).  Set \(F_i=\cl_M(Y_i)\) and \(T=F_1\cap F_2\).
Since \(F_1\cup F_2=E(M)\), Theorem~\ref{thm:intro-modularization}(1)
shows that \(F_1,F_2\) form a modular pair.  Taking closures preserves
rank, and the exact \(2\)-separation \((Y_1,Y_2)\) satisfies
\(\rk_M(Y_1)+\rk_M(Y_2)-\rk(M)=1\).  Hence
\[
 \rk_M(T)=\rk_M(F_1)+\rk_M(F_2)-\rk(M)
 =\rk_M(Y_1)+\rk_M(Y_2)-\rk(M)=1.
\]
Neither \(F_i\) is all of \(E(M)\): otherwise the other side would have
rank one, but a rank-one set in a simple matroid has at most one element,
contrary to the size condition in a \(2\)-separation.  Every rank-one flat is
modular, so Lemma~\ref{lem:recognize-modular-join} gives
\[
 M=P_T(M|F_1,M|F_2).
\]
The flats \(F_1,F_2\) are proper, and each restriction \(M|F_i\) is the
underlying matroid of a free localization.  Induction and
\cite[Theorem 4.7(ii)]{ziegler-binary} again give supersolvability.

It remains to assume that \(M\) is \(3\)-connected.  We use the following
consequences of Seymour's analysis.  A \(3\)-connected regular matroid with
an \(R_{10}\)-minor is isomorphic to \(R_{10}\).  A \(3\)-connected regular
matroid with neither an \(R_{10}\)- nor an \(R_{12}\)-minor is graphic or
cographic.
Finally, if it has an \(R_{12}\)-minor, the standard exact
\(3\)-separation \((X_1,X_2)\) of \(R_{12}\), with
\(|X_1|=|X_2|=6\), induces a \(3\)-separation
\((Y_1,Y_2)\) of \(M\) satisfying \(X_i\subseteq Y_i\).
Because \(M\) is \(3\)-connected, this induced separation is exact.
These statements are recorded in
Seymour's original decomposition analysis \cite{seymour-regular} and in
\cite[Proposition 2.12, Lemma 2.13, and Section 2.5]{mayhew-et-al}.
If \(M\) has an \(R_{10}\)-minor, the first consequence gives
\(M\cong R_{10}\), which is impossible by
Lemma~\ref{lem:free-basic-blocks}.  If \(M\) has no
\(R_{12}\)-minor, it is therefore graphic or cographic, and the same lemma
shows directly that it is supersolvable.  In the remaining case, take the
induced exact \(3\)-separation \((Y_1,Y_2)\), set
\(F_i=\cl_M(Y_i)\), and put \(T=F_1\cap F_2\).  The two sides have at
least six elements, so Theorem~\ref{thm:intro-modularization}(2) applies
and gives
\[
 M=P_T(M|F_1,M|F_2),
\]
across the three-point line \(T\), which is modular in both \(M|F_1\) and
\(M|F_2\).  The flats \(F_1,F_2\) are proper, and each restriction \(M|F_i\)
is the underlying matroid of a free localization.  Induction makes both
restrictions supersolvable, and
\cite[Theorem 4.7(ii)]{ziegler-binary} shows that \(M\) is supersolvable.
This completes the induction.

The equivalence of (2) and (3) is
Corollary~\ref{cor:binary-equivalence}, since a regular matroid is binary and
\(M(\mathcal A)\) is simple.  The equivalence of (2) and (4) is
Corollary~\ref{cor:regular-classification}.  It remains to prove
(4) \(\Rightarrow\) (1) over the original field.  If
\(M\cong M(G)\) with \(G\) chordal, then, by
\cite[Theorem 2.5]{ziegler-free}, \(\mathcal A\) is free if and only if the
standard graphic arrangement over the original field is free.  By
\cite[Theorem 1.2]{tsujie-modular}, this standard graphic arrangement is
free.
Thus (4) implies (1), and all four conditions are equivalent.
\end{proof}

\section{Consequences and a nonregular example}
\label{sec:consequences}

We finish with two consequences of the classification: a forbidden-minor and
flat-restriction criterion, and a bound on the positive exponents.  We then
recall a nonregular binary example.

\begin{corollary}
\label{cor:regular-obstructions}
Let \(M\) be a finite loopless regular matroid.  Then \(M\) admits a nice
partition if and only if the following conditions hold.
\begin{enumerate}
  \item \(M\) is simple;
  \item \(M\) has no minor isomorphic to \(M^*(K_5)\) or
  \(M^*(K_{3,3})\);
  \item there do not exist a flat \(F\) of \(M\) and an integer \(m\ge4\)
  such that \(M|F\cong M(C_m)\).
\end{enumerate}
\end{corollary}

\begin{proof}
The excluded minors for graphic matroids are
\[
 U_{2,4},\quad F_7,\quad F_7^*,\quad
 M^*(K_5),\quad M^*(K_{3,3});
\]
see \cite[Theorem 6.6.7]{oxley}.  A regular matroid excludes the first
three \cite[Theorem 6.6.6]{oxley}.  Hence, within the regular class,
condition (2) is equivalent to graphicity.

Assume that \(M\) is graphic and simple, and write \(M\cong M(G)\) with
\(G\) simple.  In a simple graph, a cycle of length at least four is induced
if and only if its edge set is a flat of the cycle matroid.  Indeed, the
closure of a set \(F\subseteq E(G)\) consists of the edges whose endpoints
lie in the same component of \((V(G),F)\).  Thus the edge set of a cycle is
a flat exactly when the cycle has no chord.  Conversely, if \(F\) is a flat
and \(M(G)|F\cong M(C_m)\), then \(F\) is an \(m\)-element circuit of
\(M(G)\), and hence the edge set of a cycle of length \(m\).  Since \(F\) is
a flat, this cycle is induced.  Therefore condition (3) is equivalent to
\(G\) being chordal.
The result follows from Corollary~\ref{cor:regular-classification}.
\end{proof}
\vspace{3mm}
\begin{corollary}
\label{cor:regular-exponent-bound}
Let \(\mathcal A\) be a finite central free arrangement of positive rank
\(r\) over an arbitrary field, and suppose that \(M(\mathcal A)\) is
regular.  The positive exponents of \(\mathcal A\) can be ordered as
\(e_1,\ldots,e_r\) so that
\[
 e_i\le i\qquad(1\le i\le r).
\]
Consequently,
\[
 |\mathcal A|=\sum_{i=1}^{r}e_i
 \le \binom{r+1}{2}.
\]
Equality holds if and only if
\(M(\mathcal A)\cong M(K_{r+1})\).
\end{corollary}

\begin{proof}
By Theorem~\ref{thm:intro-free-classification},
\(M(\mathcal A)\cong M(G)\) for a chordal simple graph \(G\).  Delete
isolated vertices, which do not affect the cycle matroid.  For each connected
component \(G_j\), write \(|V(G_j)|=r_j+1\) and choose an ordering
\(v_{j,1},\ldots,v_{j,r_j+1}\) in which the later neighbors of each vertex
form a clique \cite[Proposition 2.8]{stanley-supersolvable}.
Let \(d_{j,k}\) be the number of later neighbors of \(v_{j,k}\).
Successively deleting vertices in this order preserves connectivity until
only one vertex remains.  Hence
\[
 1\le d_{j,k}\le r_j+1-k
 \qquad(1\le k\le r_j).
\]
The characteristic-polynomial factorization for chordal graphs
\cite[Example 4.6]{stanley-supersolvable} gives
\[
 \chi(M(\mathcal A),t)=\prod_j\prod_{k=1}^{r_j}(t-d_{j,k}).
\]
By \eqref{eq:free-integral-roots}, the numbers \(d_{j,k}\) are the positive
exponents of \(\mathcal A\).  Put \(e_{j,s}=d_{j,r_j+1-s}\) for \(1\le s\le r_j\).
Then \(e_{j,s}\le s\).  When these sequences are listed component by
component, \(e_{j,s}\) occupies position \(r_1+\cdots+r_{j-1}+s\),
which is at least \(s\).  Thus \(e_i\le i\).  Summing yields
the stated bound.

If equality holds, then \(e_i=i\) for all \(i\).  There is only one component
of positive rank, since the first exponent of any later component is at most
\(1\), while its global index is greater than \(1\).  Moreover,
\(d_{1,k}=r+1-k\) for every \(k\), so every vertex is adjacent to all later
vertices and \(G=K_{r+1}\).  The converse follows from any vertex ordering of
\(K_{r+1}\).
\end{proof}

\begin{remark}[A nonregular binary example]
The regularity hypothesis in
Theorem~\ref{thm:intro-free-classification} cannot be weakened to the
assumption that the underlying matroid is binary.  Ziegler's example
\cite[Example 4.2]{ziegler-free} is as follows.  In \(\PG(4,2)\), let
\[
\begin{aligned}
M_1=\PG(4,2)\setminus\{&
 (00001),(00011),(00101),(01001),\\
 & (10001),(11001),(10101)\},\\
M_2={}&M_1\setminus\{(11110)\}.
\end{aligned}
\]
Let \(\mathcal A_i\) be the arrangement in \(\FF_2^5\) defined by the
forms \(\sum_{k=1}^5v_kx_k\) for
\(v=(v_1,\ldots,v_5)\in E(M_i)\).  Ziegler shows that
\(\mathcal A_1\) is supersolvable with
\(\exp(\mathcal A_1)=(1,2,4,8,9)\).
If \(H_p\) is the hyperplane of \(\mathcal A_1\) corresponding to
\(p=(11110)\), then \(\mathcal A_2=\mathcal A_1\setminus\{H_p\}\), while
the restriction \(\mathcal A_1^{H_p}\) has underlying matroid
\(\PG(3,2)\) and exponents \((1,2,4,8)\).  The addition--deletion theorem
therefore shows that \(\mathcal A_2\) is free with
\(\exp(\mathcal A_2)=(1,2,4,8,8)\).
Nevertheless, Ziegler proves that \(M_2\) has no modular hyperplane.  Since
a rank-five supersolvable matroid has a modular hyperplane, \(M_2\) is not
supersolvable.  
Moreover, \(M_2\) is not regular: the initial seven deleted
points all have \(x_5=1\), and \(p=(11110)\) has \(x_4=1\).
Thus the projective plane \(\{x_4=x_5=0\}\cong\PG(2,2)\) remains an
\(F_7\)-restriction of \(M_2\).  By
Corollary~\ref{cor:binary-equivalence}, \(M_2\) admits no nice partition.
\end{remark}

\section*{Declaration on the use of AI}

During the preparation of this manuscript, the authors used AI-assisted tools as auxiliary support for checking exposition, identifying possible gaps, and assisting with organization, typesetting, and English-language editing. All AI-assisted output was treated only as provisional assistance. The authors independently verified all arguments, rejected or revised incorrect or incomplete suggestions, consulted all cited sources directly, and made all final mathematical and editorial decisions. The authors have reviewed and edited the manuscript as necessary and take full responsibility for its content.

\bigskip
\begin{small}
\noindent\textsc{School of Mathematics, Hunan University, Changsha 410082, Hunan, P. R. China}\par
\noindent
\href{mailto:kangkang@hnu.edu.cn}{\texttt{kangkang@hnu.edu.cn}} (Weikang Liang)\qquad
\href{mailto:wangsuijie@hnu.edu.cn}{\texttt{wangsuijie@hnu.edu.cn}} (Suijie Wang)
\end{small}

\end{document}